%% file: heterogeneous_distributed_sensing_26.tex
\documentclass[letterpaper, 10pt, conference]{ieeeconf} \IEEEoverridecommandlockouts 
\usepackage{setspace, amsmath, amsthm, graphicx, microtype, bbold, amssymb, relsize, nag, cite}
\usepackage{todonotes}
\usepackage{hyperref}
\usepackage{cleveref}

\newcommand{\norm}[1]{\|#1\|}
\newcommand{\lap}{\mathcal{L}(\mathcal{G})}

\crefname{figure}{Figure}{Figures}
\crefname{equation}{}{}
\crefname{theorem}{Theorem}{Theorems}

\theoremstyle{definition}
\newtheorem{thm}{Theorem} 
\newtheorem{lem}{Lemma}

\newtheorem{corol}{Corollary}
\newtheorem{rem}{Remark}

\input{macro.tex}

\begin{document}
    \title{\bf Exploitation of Heterogeneity in Distributed Sensing}
	
	\author{John Daniel Peterson, Tansel Yucelen, Girish Chowdhary, and Suresh Kannan
	\thanks{J. D. Peterson is a Graduate Research Assistant of the Mechanical and Aerospace Engineering Department and a member of the Advanced Systems Research Laboratory at the Missouri University of Science and Technology, Rolla, MO 65409, USA (e-mail: {\tt\small jdp6q5@mst.edu}).
	\newline\indent T. Yucelen is an Assistant Professor of the Mechanical and Aerospace Engineering Department and the Director of the Advanced Systems Research Laboratory at the Missouri University of Science and Technology, Rolla, MO 65409, USA (e-mail: {\tt\small yucelen@mst.edu}).
	\newline\indent G. Chowdhary is an Assistant Professor of the Mechanical and Aerospace Engineering Department and the Director of the Distributed Autonomous Systems Laboratory at the Oklahoma State University, Stillwater, OK 74074, USA (e-mail: {\tt\small girish.chowdhary@okstate.edu}).
	\newline\indent S. Kannan is a Chief Scientist at the NodeIn LLC., Burlington, CT 06013, USA (e-mail: {\tt\small kannan@nodein.com}).
	\newline\indent This research was supported by the University of Missouri Research Board and the Missouri Space Grant Consortium.
	}} 
\maketitle 

\newcommand{\bequ}{\begin{eqnarray}}
\newcommand{\eequ}{\end{eqnarray}}

\begin{abstract} 

Most distributed sensing methods assume that the expected value of sensed information is same for all agents -- ignoring differences in sensor capabilities due to, for example, environmental factors and sensors' quality and condition. 
In this paper, we present a distributed sensing framework to exploit heterogeneity in information provided about a dynamic environment using an active-passive networked multiagent systems approach. 
Specifically, this approach consists of agents subject to exogenous inputs (active agents) and agents without any inputs (passive agents). 
In addition, if an active agent senses a quantity accurately (resp., not accurately), then it is weighted high (resp., low) in the network such that these weights can be a function of time due to varying environmental factors. 
The key feature of our approach is that the states of all agents converge to an adjustable neighborhood of the weighted average of the sensed exogenous inputs by the active agents. 

\end{abstract}


\vspace{0.1cm}

\section{Introduction}
\label{sec:intro}


Classical distributed sensing methods assume instantaneous communication, which is not practical for situations involving a large number of agents, high-dimensional measurements, and unpredictable low-bandwidth networks \cite{spanos2005distributed,spanos2005approximate,mesbahi2010graph}. 
Unlike classical methods, system-theoretical distributed sensing approaches involve equations of motion to describe dynamic behavior of the information fusion process, which allow one to understand overall network behavior, and can also have better robustness to uncertainties (e.g., asynchronous operations, time-varying link availability, and measurement noise) \cite{olfati2004consensus,spanos2005distributed,spanos2005approximate,olfati2005consensus,kar2007distributed,kar2009distributed,yucelen2012control}. 
A key component of system-theoretical distributed sensing is a consensus algorithm needed for the information fusion process. 

Among two widely-used classes of consensus algorithms, static and dynamic consensus algorithms, dynamic ones consider agreement upon time-varying quantities and are well-suited for dynamic environment applications. 
Existing dynamic consensus algorithms are suitable for applications where each agent is subject to one input measurement \cite{spanos2005dynamic,olfati2005consensus,yang2006stability,bai2010robust,taylor2011dynamic,chen2012distributed}. 
However, {network nodes may differ in the number of input measurements}; for example, one agent may not sense a quantity and another may sense multiple quantities for certain time instants. 
While \cite{ustebay2009selective,ustebay2011efficient,mu2013efficient,Mu2014} present methods that cover applications when a portion of the agents do not perform sensing (they still assume that the rest of the agents are subject to one input measurement only),  \cite{ustebay2009selective,ustebay2011efficient} consider static consensus algorithms (i.e., not suitable for dynamic environments) and \cite{mu2013efficient,Mu2014} consider homogenous sensing capability across nodes. 

It is important to point out that sensing capability of each agent, measured by the \textit{value of information}, may not be the same 
for all agents due to variations in quality and condition of sensors and/or simply distance to the available information, and therefore, 
some agents can have better sensing power and less sensing error than others. 
Consequently, {heterogeneity in sensing capability needs to be considered to achieve reliable and correct network performance}. 
Even though there exist a few works \cite{martinez2006optimal,yucelen2012control,shge2014,olfati2012coupled,mu2013efficient,Mu2014} 
that consider the value of information, these results either deal with analysis in the context of static consensus 
algorithms (e.g., assume that location of nodes and targets remain constant) or consider mission-specific scenarios 
(e.g., sensor placement, target tracking, or information broadcasting). 


In this paper, we present a distributed sensing framework to exploit 
heterogeneity in information provided about a dynamic environment using an active-passive networked multiagent systems approach. 
This approach is introduced in \cite{ap_networks} to remove the assumption that each agent is subject to one input measurement, 
which is common among the class of dynamic consensus algorithms.
In particular, the approach in \cite{ap_networks} consists of agents subject to exogenous inputs (active agents) and agents 
without any inputs (passive agents), where these inputs may or may not overlap within the active agents. 
This paper utilizes and generalizes the active-passive networked multiagent systems approach of \cite{ap_networks} to account for
heterogeneity in agents' sensing capability measured by the value of information.
Specifically, if an active agent senses a quantity accurately (resp., not accurately), 
then it reports a high (resp., low) value of information, and hence,  it is weighted high (resp., low) in the network as compared to other agents. 
In addition, these weights can be a function of time due to varying environmental factors and/or changes in sensors' quality and condition. 
The key feature of our approach is that the states of all agents converge to an adjustable neighborhood of the weighted average of the sensed
exogenous inputs by the active agents with heterogeneous sensing capabilities. 


\section{Mathematical Preliminaries}
\label{sec:prelim}

The notation used in this paper is fairly standard (see, for example, \cite{ap_networks}). 
We first recall some of the basic notions from graph theory and refer to \cite{mesbahi2010graph,ref:001} for further details.
In the multiagent literature, graphs are broadly adopted to encode interactions in networked systems.
An \textit{undirected} graph $\mathcal{G}$ is defined by a set $\mathcal{V}_{\mathcal{G}}=\{1,\ldots,n\}$ of \textit{nodes}
and a set $\mathcal{E}_{\mathcal{G}} \subset \mathcal{V}_{\mathcal{G}} \times \mathcal{V}_{\mathcal{G}}$ of \textit{edges}.
If $(i,j) \in \mathcal{E}_{\mathcal{G}}$, then the nodes $i$ and $j$ are \textit{neighbors} and the neighboring relation is indicated with $i \sim j$.
The \textit{degree} of a node is given by the number of its neighbors.
Letting $d_i$ be the degree of node $i$, then the \textit{degree} matrix of a graph $\mathcal{G}$, $\mathcal{D}(\mathcal{G}) \in \IR^{n \times n}$, is given by $\mathcal{D}(\mathcal{G}) \triangleq \diag(d), \ d=[d_1,\ldots,d_n]\mT$.
A \textit{path} $i_0 i_1 \ldots i_L$ is a finite sequence of nodes such that $i_{k-1} \sim i_k$, $k=1, \ldots, L$,
and a graph $\mathcal{G}$ is \textit{connected} if there is a path between any pair of distinct nodes.
The \textit{adjacency} matrix of a graph $\mathcal{G}$, $\mathcal{A}(\mathcal{G}) \in \IR^{n \times n}$, is given by

\begin{equation}
   [\mathcal{A}(\mathcal{G})]_{ij} \triangleq
   \left\{ \begin{array}{cl}
      1, &\mbox{ if $(i,j)\in\mathcal{E}_\mathcal{G}$},\\
      0, &\mbox{otherwise}.
   \end{array} \right.
   \label{AdjMat}
\end{equation}
The \textit{Laplacian} matrix of a graph, $\mathcal{L}(\mathcal{G}) \in \overline{\IS}_+^{\hspace{0.1em} n \times n}$, 
playing a central role in many graph theoretic treatments of multiagent systems, is given by 
$\mathcal{L}(\mathcal{G}) \triangleq \mathcal{D}(\mathcal{G}) - \mathcal{A}(\mathcal{G})$. 
Throughout this paper, we model a given multiagent system by a connected, undirected graph $\mathcal{G}$, 
where nodes and edges represent agents and inter-agent communication links, respectively.


Next, we introduce several necessary lemmas used in the main results of this paper. 

\begin{lem}[\hspace{-0.01cm}{\cite{mesbahi2010graph}}] 
	 The spectrum of the Laplacian of a connected, undirected graph can be ordered as 
	\begin{equation}
		0 = \lambda_1(\mathcal{L}(\mathcal{G}))<\lambda_2(\mathcal{L}(\mathcal{G}))\le \cdots \le \lambda_n(\mathcal{L}(\mathcal{G})), \label{LapSpec}
	\end{equation}
	with $\one_n$ as the eigenvector corresponding to the zero eigenvalue $\lambda_1(\mathcal{L}(\mathcal{G}))$ and 
	$\mathcal{L}(\mathcal{G}) \one_n = \zero_n$ and $\rom{e}^{\mathcal{L}(\mathcal{G})\one_n} = \one_n$. 
	\label{lem:1}
\end{lem}

\begin{lem}[\hspace{-0.01cm}{\cite{g:inv}}]
	The Laplacian of a connected, undirected graph satisfies $\mathcal{L}(\mathcal{G}) \mathcal{L}^\dagger(\mathcal{G})=\eye_n-\frac{1}{n}\one_n\one_n\mT$. 
\end{lem}

\vspace{-0.3cm}

\begin{lem}
	Let $K=\rom{diag}(k)$, $k=[k_1, k_2, \ldots, k_n]\mT$, $k_i \in \overline{\mathbb{R}}_+$, $i=1,\ldots,n$, and assume that at least one element of $k$ is nonzero. 
	Then, for the Laplacian of a connected, undirected graph, 
	\bequ
	\mathcal{F}(\mathcal{G})\triangleq\mathcal{L}(\mathcal{G})+K\in\IS_+^{n \times n}, \label{matrix_f}
	\eequ
	and $\rom{det}(\mathcal{F}(\mathcal{G}))\neq0$.
	\label{lem:2}
\end{lem}

\begin{proof} 
Consider the decomposition $K=K_1+K_2$, where $K_1\triangleq\rom{diag}([0, \ldots, 0, \phi_i, 0, \ldots, 0]\mT)$ and $K_2\triangleq K-K_1$, where $\phi_i$ denotes the smallest nonzero diagonal element of $K$ appearing on its $i$-th diagonal, so that $K_2\in\overline{\IS}_+^{n \times n}$. 
From the Rayleigh's Quotient \cite{Lay2006}, the minimum eigenvalue of $\mathcal{L}(\mathcal{G})+K_1$ can be given by
\bequ
\lambda_{\text{min}}(\mathcal{L}(\mathcal{G})+K_1) &\hspace{-0.1cm}=\hspace{-0.1cm}& \min\limits_{x} \{x\mT\bigl(\mathcal{L}(\mathcal{G})+K_1\bigl)x \: | \: x\mT x = 1\},\nonumber\\ \label{added:equation:1}
\eequ
where $x$ is the eigenvector corresponding to this minimum eigenvalue. 
Note that since $\mathcal{L}(\mathcal{G})\in\overline{\IS}_+^{n \times n}$ and $K_1\in\overline{\IS}_+^{n \times n}$, and hence, $\mathcal{L}(\mathcal{G})+K_1$ is real and symmetric, $x$ is a real eigenvector. 
Now, expanding (\ref{added:equation:1}) as
\bequ
		 x\mT\left( \mathcal{L}(\mathcal{G}) + K_1\right)x &=& \sum\limits_{i \thicksim j}a_{ij}(x_i-x_j)^2 + \phi_i x_i^2,
		\label{eq:lem2}
\eequ
and noting that the right hand side of (\ref{eq:lem2}) is zero only if $x\equiv0$, it follows that $\lambda_{\text{min}}(\mathcal{L}(\mathcal{G}) + K_1) > 0$, and hence, $\mathcal{L}(\mathcal{G}) + K_1\in\IS_+^{n \times n}$. 
Finally, let $\lambda$ be an eigenvalue of $\mathcal{F}(\mathcal{G})=\mathcal{L}(\mathcal{G}) + K_1 + K_2$. 
Since $\lambda_\rom{min}(\mathcal{L}(\mathcal{G}) + K_1)>0$ and $\lambda_\rom{min}(K_2)=0$, it follows from Fact 5.11.3 of \cite{bernstein2009matrix} that $\lambda_\rom{min}(\mathcal{L}(\mathcal{G}) + K_1) + \lambda_\rom{min}(K_2) \le \lambda$, 
and hence, $\lambda>0$, which implies that (\ref{matrix_f}) holds and $\rom{det}(\mathcal{F}(\mathcal{G}))\neq0$. 
\end{proof}


\section{Overview of Active-Passive Networked Multiagent Systems}
\label{sec:apover}
In this section, we briefly overview the active-passive networked multiagent systems approach of \cite{ap_networks}. 
In particular, consider a system of $n$ agents exchanging information among each other using their local measurements according to a connected, undirected graph $\mathcal{G}$. 
In addition, consider that there exists $m\ge1$ exogenous inputs that interact with this system. 

\vspace{0.1cm}

\noindent \textbf{Definition 1.} If agent $i$, $i=1,\ldots,n$, is subject to one or more exogenous inputs (resp., no exogenous inputs), then it is an active agent (resp., passive agent). 

\vspace{0.1cm}

\noindent \textbf{Definition 2.} If an exogenous input interacts with only one agent (resp., multiple agents), then it is an isolated input (resp., non-isolated input).

\vspace{0.1cm}

The approach presented in \cite{ap_networks} deals with the problem of driving the states of all (active and passive) agents to the average of the applied exogenous inputs.  
For this purpose, the following integral action-based distributed sensing algorithm is proposed
\vspace{0cm}
\begin{align}
	\dot{x}_i(t) = &-\alpha\sum\limits_{i \sim j}\left(x_i(t)-x_j(t)\right)+\sum\limits_{i \sim j}\left(\xi_i(t)-\xi_j(t)\right)  \nonumber\\
	 &-\alpha\sum\limits_{i \sim h}\left(x_i(t)-c_h(t)\right), x_i(0)=x_{i0}, \label{eq:ap_orig_1}\\
	\dot{\xi}_i(t) = &-\gamma\sum\limits_{i \sim j}\left(x_i(t)-x_j(t)\right), \xi_i(0)=\xi_{i0}, \ \ \ \ \ \ \label{eq:ap_orig_2}
\end{align}
where $x_i(t)\in\IR$ and $\xi_i(t)\in\IR$ denote the state and the integral action of agent $i$, $i=1,\ldots,n$, respectively, $c_h(t)\in\IR$, $h=1,\ldots,m$, denotes an exogenous input sensed by this agent, $\alpha\in\IR_{+}$, and $\gamma \in \IR_{+}$.
Note that $i \sim h$ notation indicates the exogenous inputs that an agent is subject to, which is similar to the $i \sim j$ notation indicating the neighboring relation between agents. 

\vspace{0.1cm}

\begin{rem}
Theorem 1 of \cite{ap_networks} shows that the states of all agents converge to the average of the exogenous inputs applied to active agents under the assumption that all active agents have the same value of sensed information, and hence, are weighted identically.
\end{rem}

\vspace{0.1cm}

\begin{rem}
The proposed integral action-based distributed algorithm in (\ref{eq:ap_orig_1}) is applied to agents having dynamics of the form $\dot{x}_i(t)=u_i(t)$, where $u_i(t)\in\IR$ denotes the input of agent $i$, $i=1,\ldots,n$, satisfying the right hand side of (\ref{eq:ap_orig_1}) along with (\ref{eq:ap_orig_2}). 
For agents having complex dynamics, one can design low-level feedback controllers (or assume their existence) for suppressing existing dynamics and enforcing $\dot{x}_i(t)=u_i(t)$ (see, for example, Example 6.3 of \cite{yucelen2014consensus}). 
\end{rem}


\section{Exploitation of Heterogeneity}

The value of sensed information is not necessarily identical for all active agents due to environmental factors and/or sensors' quality and condition, as previously discussed. 
In this section, we generalize the active-passive networked multiagent systems to account for heterogeneity in active agents' sensing capability. 


\subsection{Problem Setup} 

We begin with proposing the following integral action-based distributed sensing algorithm
\begin{align}
	\dot{x}_i(t) = &-\alpha\sum\limits_{i \sim j}\left(x_i(t)-x_j(t)\right)+\sum\limits_{i \sim j}\left(\xi_i(t)-\xi_j(t)\right)  \nonumber\\
	 &-\alpha\sum\limits_{i \sim h}w_{ih}(t)\left(x_i(t)-c_{h}(t)\right), \ \ x_i(0)=x_{i0}, \label{eq:ap_weight_1}\\
	\dot{\xi}_i(t) = &-\gamma\left[\sum\limits_{i \sim j}\left(x_i(t)-x_j(t)\right) +\sigma\xi_i(t)\right], \ \ \xi_i(0)=\xi_{i0}, 
	\label{eq:ap_weight_2}
\end{align} 
where $x_i(t)\in\IR$ and $\xi_i(t)\in\IR$ denote the state and the integral action of agent $i$, $i=1,\ldots,n$, respectively, 
$c_h(t)\in\IR$, $h=1,\ldots,m$, denotes an exogenous input sensed by this agent, 
$\alpha\in\IR_{+}$, $\gamma \in \IR_{+}$, and $\sigma \in \mathbb{R}_+$. 
Note that $w_{ih}(t) \in \IR_+$ is a weight capturing the expected value of information of the exogenous input with respect to agent $i$. 

\vspace{0.1cm}

\begin{rem}
The focus of this paper is to develop a distributed sensing framework to exploit 
heterogeneity in sensed information. For this reason, we implicitly assume that the value of sensed information is already modeled and 
known by respective agents in the network. Note that each sensor can obtain the value of information, for example, through analysis of its own 
measurement using relative entropy measures such as Kullback-Liebler divergence \cite{Mu2014}. 
\end{rem}

Next, let 
\bequ
x(t)&=&\bigl[x_1(t),x_2(t),\ldots,x_n(t)\bigl]\mT\in\IR^n, \label{oldpaper:1} \\
\xi(t)&=&\bigl[\xi_1(t),\xi_2(t),\ldots,\xi_n(t)\bigl]\mT\in\IR^n, \label{oldpaper:2} \\
c(t)&=&\bigl[c_1(t),c_2(t),\ldots,c_m(t),0,\ldots,0\bigl]\in\IR^n, \label{oldpaper:3} \ \ \ \
\eequ 
where $m \le n$ is assumed to ease notation without loss of generality. 
We can now rewrite (\ref{eq:ap_weight_1}) and (\ref{eq:ap_weight_2}) in the compact form given by
\vspace{0cm}
\bequ
	\dot{x}(t) =& -\alpha\lap x(t) + \lap \xi(t)-\alpha K_1(t)x(t) \nonumber\\
	&  + \alpha K_2(t)c(t), \quad x(0)=x_0, \label{eq:ap_weight_c_1}\eequ\bequ
	\dot{\xi}(t) =& -\gamma \lap x(t) - \gamma \sigma \xi(t), \quad \xi(0)=\xi_0,  \label{eq:ap_weight_c_2}
\eequ
where $\mathcal{L}(\mathcal{G}) \in \overline{\IS}_+^{\hspace{0.1em} n \times n}$ satisfies Lemma 1, 
\bequ
	K_1(t)&\triangleq&\rom{diag}([k_{1,1}(t),\ldots,k_{1,n}(t)]\mT)\in\overline{\IS}_+^{\hspace{0.1em} n \times n}, \ \ \ \
\eequ 
with $k_{1,i}\in\overline{\IR}_+$ denoting the number of the exogenous inputs applied to agent $i$, $i=1,\ldots,n$, and
\vspace{0.1cm}
\bequ
 K_2(t)&\triangleq&\begin{bmatrix}
 k_{2,11}(t) & \cdots & k_{2,1n}(t)\\
 k_{2,21}(t) & \cdots & k_{2,2n}(t) \\
  \vdots &  \ddots &\vdots \\
  k_{2,n1}(t)&\cdots & k_{2,nn}(t)
 \end{bmatrix}\in\IR^{n \times n}, \ \ \ \
\eequ
with \vspace{-0.2cm}
\bequ
k_{1,i}(t)&=&\sum_{j=1}^{n}k_{2,ij}(t). \label{hello:world:2}
\eequ \vspace{-0.6cm}

\begin{figure}[t!] \hspace{1.1cm} \includegraphics[scale=0.5]{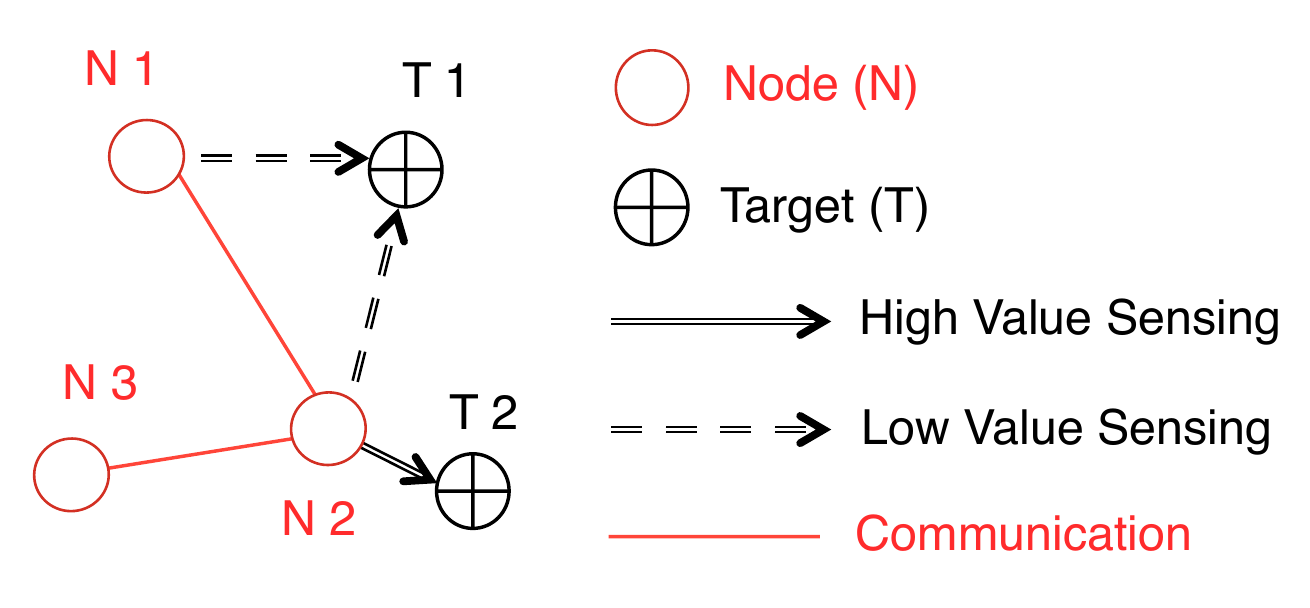} \vspace{-0.5cm}
\caption{Active-passive networked multiagent system with two targets and three agents with nodes 1 and 2 being active and node 3 being passive.} \vspace{-1cm}
\label{top:55}
\end{figure}

\begin{rem}
In (\ref{eq:ap_weight_c_1}) and (\ref{eq:ap_weight_c_2}), the elements of $K_1(t)$ and $K_2(t)$ are related to the weights $w_{ih}(t) \in \IR_+$ capturing the expected value of information. 
To elucidate this point, consider the active-passive networked multiagent system given in Figure \ref{top:55}. 
Let $c=[c_1,c_2,0]\mT$, where 
$c_1=f(q_\rom{T_1},||p_{x_1}-p_{\rom{T}_1}||_2)=f(q_\rom{T_1},||p_{x_2}-p_{\rom{T}_1}||_2)$ (assuming $||p_{x_1}-p_{\rom{T}_1}||_2=||p_{x_2}-p_{\rom{T}_1}||_2$) and $c_2=f(q_\rom{T_2},||p_{x_2}-p_{\rom{T}_2}||_2)$, 
with $q_{\rom{T}_i}$, $i=1,2$, representing the actual target quantities 
and $p_{x_i}$, $i=1,2,3$, representing the location of nodes. 
In this case, if $w_{11}(t)$, $w_{21}(t)$, and $w_{22}(t)$ respectively denote the value of information of input $c_1$ with respect to agent 1, input $c_1$ with respect to agent 2, and input $c_2$ with respect to agent 2 (note that $w_{11}(t)=w_{21}(t)$), then one can write
\bequ
K_1(t)&=&\begin{bmatrix} w_{11}(t) & 0 & 0 \\ 0 & w_{21}(t)+w_{22}(t) & 0 \\ 0 & 0 & 0 \end{bmatrix}, \\
K_2(t)&=&\begin{bmatrix} w_{11}(t) & 0 & 0 \\ w_{21}(t) & w_{22}(t) & 0 \\ 0 & 0 & 0 \end{bmatrix}.  
\eequ
As an another example, consider the case $||p_{x_1}-p_{\rom{T}_1}||_2\neq||p_{x_2}-p_{\rom{T}_1}||_2$ and let $c=[c_1,c_2,c_3]\mT$, 
where $c_1=f(q_\rom{T_1},||p_{x_1}-p_{\rom{T}_1}||_2)$, $c_2=f(q_\rom{T_1},||p_{x_2}-p_{\rom{T}_1}||_2)$, and $c_3=f(q_\rom{T_2},||p_{x_2}-p_{\rom{T}_2}||_2)$. 
In this case, if $w_{11}(t)$, $w_{22}(t)$, and $w_{23}(t)$ respectively denote the value of information of input $c_1$ with respect to agent 1, input $c_2$ with respect to agent 2, and input $c_3$ with respect to agent 2, then one can write
\bequ
K_1(t)&=&\begin{bmatrix} w_{11}(t) & 0 & 0 \\ 0 & w_{22}(t)+w_{23}(t) & 0 \\ 0 & 0 & 0 \end{bmatrix}, \eequ\bequ
K_2(t)&=&\begin{bmatrix} w_{11}(t) & 0 & 0 \\ 0 & w_{22}(t) & w_{23}(t) \\ 0 & 0 & 0 \end{bmatrix}.  
\eequ
\end{rem}

Without loss of generality, assume that $k_{2,ij}\in \left[0,1\right]$, which corresponds to the fact that each value of information can take values from the scaled interval $\left[0,1\right]$ (i.e., as agents' value of information increase about exogenous inputs, $k_{2,ij}$ increases from $0$ toward $1$). 
Since we are interested in driving the states of all (active and passive) agents to an adjustable neighborhood of the weighted average of the exogenous inputs applied to the active agents, let
\bequ
	\delta(t) &\triangleq& x(t)-\epsilon(t)\one_n \in \mathbb{R}^n, \label{eq:error_var_1} \\
	\epsilon(t) &\triangleq& \frac{\one_n^T K_2(t) c(t)}{\one_n^T K_2(t) \one_n} \in \mathbb{R}, \label{eq:error_var_2}
\eequ
be the error between $x_i(t), i=1, \ldots,n$, and the weighted average of the applied
exogenous inputs $\epsilon(t)$. 
Based on (\ref{eq:error_var_2}),  $\epsilon(t)$ can be equivalently written as
\begin{align}
	\epsilon(t)= & \bigl( k_{2,11}(t)c_{1}(t) + k_{2,12}(t)c_{2}(t) + \cdots + k_{2,21}(t)c_{1}(t) \nonumber \\ 	
	&+ k_{2,22}(t)c_{2}(t) + \cdots \bigr) / \bigl( k_{2,11}(t) + k_{2,12}(t) \nonumber \\ 	
	&+ \cdots + k_{2,21}(t) + k_{2,22}(t) + \cdots \bigr). \label{eq:error_var_3}
\end{align}
Note that the denominator of (\ref{eq:error_var_3}) is nonzero, since we assume that there exists $m\ge1$ exogenous inputs, and hence, there exists at least one nonzero value of information weight. 
Furthermore, let this nonzero weight be lower bounded by $\phi_i \in \IR_+$ and consider the decomposition 
\bequ
K_1(t)=K_0+\tilde{K}(t), \label{decom:3}
\eequ 
where $K_0\triangleq\rom{diag}([0, \ldots, 0, \phi_i, 0, \ldots, 0]\mT)$ and $\tilde{K}(t)\triangleq K_1(t)-K_0$ such that $\tilde{K}\in\overline{\IS}_+^{n \times n}$. 
Note that for situations where there exists more than one nonzero value of information weights, then we can either include the lower bounds of these nonzero weights to $K_0$ or we simply let $\phi_i$ to represent the smallest lower bound of these nonzero weights, and hence, we can  perform this decomposition. 
This concludes the setup of our problem. 
Next, we present the stability and performance guarantees of the distributed sensing algorithm given by (\ref{eq:ap_weight_1}) and (\ref{eq:ap_weight_2}). 


\subsection{Stability and Performance Guarantees} 

Consider the error between $x_i(t)$, $i=1,\ldots,n$, and the weighted average of the sensed exogenous inputs $\epsilon(t)$ given by (\ref{eq:error_var_1}). 
Using Lemma 1, the time derivative of (\ref{eq:error_var_1}) can be given by
\bequ
	\dot{\delta}(t)&=&-\alpha \tilde{F}(\mathcal{G},t)\delta(t)+\mathcal{L}(\mathcal{G})\xi(t)-\alpha L_c(t) K_2(t) c(t)\nonumber\\
	&& -\dot{\epsilon}(t)\one_n, \quad \delta(0)=\delta_0, \label{tansel:01}
\eequ
where 
$\tilde{F}(\mathcal{G},t) \triangleq \mathcal{L}(\mathcal{G}) + K_1(t),$ and 
$L_c(t) \triangleq \frac{K_1(t) \one_n \one_n\mT}{\one_n\mT K_2(t) \one_n} - \eye_n$.
Now consider
\bequ
	e(t) &\triangleq& \xi(t)-\alpha\mathcal{L}^\dagger (\mathcal{G}) L_c(t) K_2(t) c(t). \label{tansel:04}
\eequ
Using Lemma 2 and noting that $\one_n\mT L_c(t) = \one_n\mT\big[{K_1(t) \one_n \one_n\mT}/({\one_n\mT K_2(t) \one_n}) - \eye_n\bigl]  = 0$, (\ref{tansel:01}) can be rewritten by
\bequ
	\dot{\delta}(t)&=&-\alpha \tilde{F}(\mathcal{G},t)\delta(t)+\mathcal{L}(\mathcal{G})e(t)-\dot{\epsilon}(t)\one_n.  \label{tansel:06}
\eequ
Furthermore, applying the decomposition (\ref{decom:3}) to (\ref{tansel:06}), it follows that
\bequ
	\dot{\delta}(t)&\hspace{-0.1cm}=\hspace{-0.1cm}&-\alpha {F}(\mathcal{G})\delta(t)-\alpha \tilde{K}(t)\delta(t)+\mathcal{L}(\mathcal{G})e(t)-\dot{\epsilon}(t)\one_n, \nonumber\\ \label{tansel:07}
\eequ
where ${F}(\mathcal{G}) \triangleq \mathcal{L}(\mathcal{G}) +K_0 \in \IS_+^{n \times n}$ is a direct consequence of Lemma 3. 
Finally, the time derivative of (\ref{tansel:04}) can be given by
\bequ
	\dot{e}(t)&\hspace{-0.1cm}=\hspace{-0.1cm}&-\gamma  \mathcal{L}(\mathcal{G}) \delta(t) - \gamma \sigma e(t) -\alpha\gamma\sigma \mathcal{L}^\dagger(\mathcal{G})K_c(t)c(t) \nonumber\\ &&- \alpha\mathcal{L}^\dagger(\mathcal{G})\dot{K}_c(t)c(t)- \alpha\mathcal{L}^\dagger(\mathcal{G})K_c(t) \dot{c}(t), \nonumber\\  && \hspace{4.325cm} e(0)=e_0,  \label{tansel:09}
\eequ
where $K_c(t) \triangleq L_c(t) K_2(t)$. 

Next, the closed-loop error dynamics given by (\ref{tansel:07}) and (\ref{tansel:09}) can be rewritten as
\begin{align}
	\dot{\delta}(t)=&-\alpha {F}(\mathcal{G})\delta(t)-\alpha \tilde{K}(t)\delta(t)+\mathcal{L}(\mathcal{G})e(t)+s_1(t), \nonumber\\ \label{closed:01} \\
	\dot{e}(t)=&-\gamma  \mathcal{L}(\mathcal{G}) \delta(t) - \gamma \sigma e(t) + s_2(t), \label{closed:02}
\end{align}
where the perturbation terms are given by 
\bequ
s_1(t)&\triangleq&-\dot{\epsilon}(t)\one_n, \\
s_2(t)&\triangleq&-\alpha\gamma\sigma \mathcal{L}^\dagger(\mathcal{G})K_c(t)c(t) - \alpha\mathcal{L}^\dagger(\mathcal{G})\dot{K}_c(t)c(t) \nonumber\\&& - \alpha\mathcal{L}^\dagger(\mathcal{G})K_c(t) \dot{c}(t). 
\eequ
For the following results, we assume $||s_1(t)||_2 \le s_1^*$ and $||s_2(t)||_2 \le s_2^*$.

\vspace{0.1cm}

\begin{thm}
	Consider the networked multiagent system given by (\ref{eq:ap_weight_1}) and (\ref{eq:ap_weight_2}), where agents exchange 
	information using local measurements through a connected and undirected graph topology. 
	Then, the closed-loop error dynamics given by (\ref{closed:01}) and (\ref{closed:02}) are bounded. 
	\label{thm:main}
\end{thm}
\begin{proof}
	The result follows by considering the Lyapunov function candidate given by  
	\begin{equation}
		V(\delta,e)=\frac{1}{2\alpha}\delta\mT\delta+\frac{1}{2\alpha\gamma}e\mT e, \label{prfx:1hey} 
	\end{equation}
    and differentiating it along the closed loop error dynamics given by \cref{closed:01} and \cref{closed:02}.
\end{proof}
\begin{rem}
	From \Cref{thm:main}, one can calculate an estimate of the ultimate bound for $\delta(t)$, $t \ge T$ as
	\begin{align}
		\norm{\delta(t)}^2_{2}\le& \frac{1}{\alpha^2}\biggl[\frac{n^2 \dot{\epsilon}^{*2}}{\lambda_\rom{min}^2(\mathcal{F}(\mathcal{G}))}\bigg] + \frac{\alpha^2}{\gamma}\biggl[  p_1^{*2}  \nonumber\\ 
		& +\frac{2 p_1^* p_2^*}{\gamma\sigma} + \frac{p_2^{*2}}{\gamma^2 \sigma^2} \bigg], \ \ \ \ \  \label{BOUND:GIANT}
	\end{align}
	where $||\dot{\epsilon}(t)||_2\le\dot{\epsilon}^*$, $||\mathcal{L}^\dagger(\mathcal{G})K_c(t)c(t)||_2\le p_1^*$, 
	and $|| \mathcal{L}^\dagger(\mathcal{G})\dot{K}_c(t)c(t) +\mathcal{L}^\dagger(\mathcal{G})K_c(t) \dot{c}(t) ||_2\le p_2^*$.
	This implies that if we choose $\alpha$ and $\gamma$ such that both $1/\alpha^2$ and $\alpha^2/\gamma$ are small,
	then \cref{BOUND:GIANT} is small for $t \ge T$. 
\end{rem}

\subsection{Special Case Corollaries} 

In the previous section, it is shown that the states of all agents can be driven to an adjustable neighborhood 
of the weighted average of the sensed exogenous inputs by the active agents for the case of time-varying sensed 
exogenous inputs, which is due to a dynamic environment, with time-varying value of information. We now give several
corollaries that present special cases of \Cref{thm:main}.
We begin with the case of time-varying sensed exogenous inputs and constant value of information. 

\vspace{0.1cm}

\begin{corol}
	Consider the networked multiagent system given by (\ref{eq:ap_weight_1}) and (\ref{eq:ap_weight_2}), where agents exchange information using local measurements through a connected and undirected graph topology. 
	Let the value of information be constant. 
	Then, the closed-loop error dynamics given by (\ref{closed:01}) and (\ref{closed:02}) are bounded and the bound of $\delta(t)$ for $t \ge T$ is given by (\ref{BOUND:GIANT}) with $\dot{\epsilon}^* = \dot{c}^*||\one_n\mT K_2||_2 / (\one_n\mT K_2 \one_n)$, $p_1^* =  c^*||\mathcal{L}^\dagger(\mathcal{G})K_c||_\rom{F}$, and $p_2^* = \dot{c}^*||\mathcal{L}^\dagger(\mathcal{G})K_c||_\rom{F}$, 
	where $||c(t)||_2\le c^*$ and $||\dot{c}(t)||_2 \le \dot{c}^*$. 
\end{corol}

The next corollary presents the case of constant sensed exogenous inputs and time-varying value of information. 

\begin{corol}
	Consider the networked multiagent system given by (\ref{eq:ap_weight_1}) and (\ref{eq:ap_weight_2}), where agents exchange information using local measurements through a connected and undirected graph topology. 
	Let the sensed exogenous inputs be constant. 
	Then, the closed-loop error dynamics given by (\ref{closed:01}) and (\ref{closed:02}) are bounded and the bound of $\delta(t)$ for $t \ge T$ is given by (\ref{BOUND:GIANT}) with $\dot{\epsilon}^* \ge || {\one_n\mT K_2(t) \bigl[\one_n \one_n\mT \dot{K}_2(t) c - c \one_n\mT \dot{K}_2(t) \one_n\bigl]}/{(\one_n\mT K_2(t) \one_n)^2}||_2$, $p_1^* =  ||\mathcal{L}^\dagger(\mathcal{G})||_\rom{F} ||c||_2 k_c^*$, and $p_2^* = ||\mathcal{L}^\dagger(\mathcal{G})||_\rom{F} ||c||_2 \dot{k}_c^*$, 
	where $||K_2(t)||_\rom{F}\le k_c^*$ and $||\dot{K}_2(t)||_\rom{F}\le \dot{k}_c^*$. 
\end{corol}

The bounds included in Corollaries 1 and 2 clearly show how the time rate of change of the exogenous inputs and the value of information affect (\ref{BOUND:GIANT}), respectively. 
We now present the final case of constant sensed exogenous inputs and constant value of information. 

\vspace{0.1cm}

\begin{corol}
	Consider the networked multiagent system given by (\ref{eq:ap_weight_1}) and (\ref{eq:ap_weight_2}) with $\sigma=0$, where agents exchange information 
	using local measurements through a connected and undirected graph topology. 
	Let both the value of information and the sensed exogenous inputs be constant. 
	Then, the closed-loop error dynamics given by (\ref{closed:01}) and (\ref{closed:02}) are Lyapunov stable for all initial conditions and $\delta(t)$ asymptotically vanishes.
\end{corol}


\section{Illustrative Numerical Examples} \label{sec:apweights_observe}

\begin{figure}[t!] \hspace{1.1cm} \includegraphics[scale=0.5]{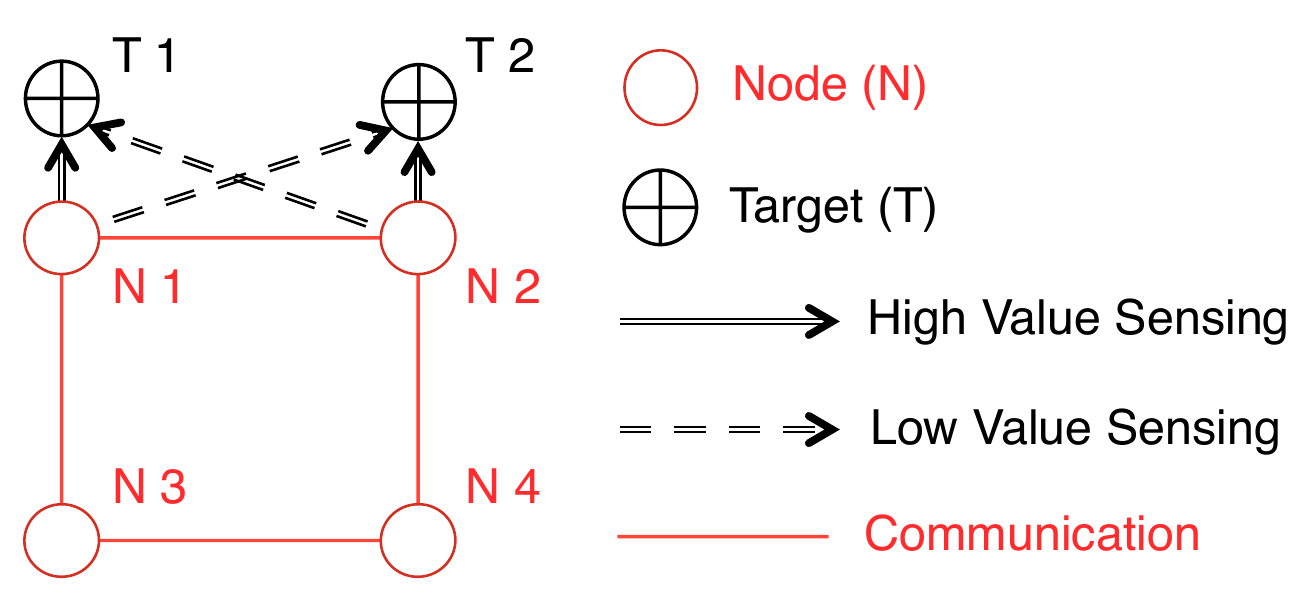} \vspace{-0.4cm}
\caption{Active-passive networked multiagent system with two targets and four agents with nodes 1 and 2 being active and nodes 3 and 4 being passive.} \vspace{-0.325cm}
\label{illustrative:01}
\end{figure}

\begin{figure}[t!] \hspace{-0.1cm} \includegraphics[scale=0.486]{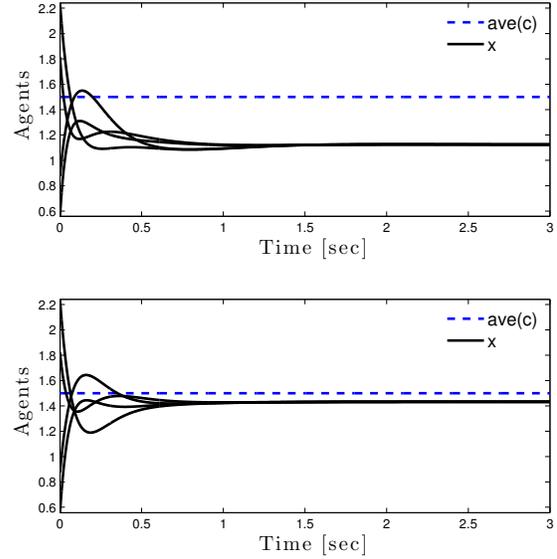} \vspace{-0.4cm}
\caption{Response of the networked multiagent system in Figure \ref{illustrative:01} with identical value of information (top) and heterogeneous value of information (bottom) (dashed lines denote the actual average of the target quantities and solid lines denote the agent states).} \vspace{-0.35cm}
\label{illustrative:02}
\end{figure}

In this section, the efficacy of the proposed distributed sensing algorithm given by (\ref{eq:ap_weight_1}) and (\ref{eq:ap_weight_2}) is illustrated using two examples.
For the first example, consider the active-passive networked multiagent system given in Figure \ref{illustrative:01}. 
Let the value of information and the exogenous inputs be constant, and the accuracy of the sensed target quantity decreases (resp., increases) as the distance between a node and the target gets larger (resp., smaller). 
In addition, let nodes 1 and 2 be sense targets 1 and 2 with perfect accuracy, respectively,  
and sense targets 2 and 1 with $50\%$ accuracy, respectively. 
For this purpose, we set $c_1=1$ (node 1 measure of target 1 with perfect accuracy), $c_2=1$ (node 1 measure of target 2 with $50\%$ accuracy), $c_3=0.5$ (node 2 measure of target 1 with $50\%$ accuracy), $c_4=2$ (node 2 measure of target 2 with perfect accuracy). 
Figure \ref{illustrative:02} presents the results with $\alpha=5$, $\gamma=10$, and $\sigma=0$ in (\ref{eq:ap_weight_1}) and (\ref{eq:ap_weight_2}). 
Specifically, the top figure shows the network response with identical value of information (i.e., we set $w_{11}=w_{12}=w_{23}=w_{24}=1$) and the bottom figure shows the same response with heterogeneous value of information (i.e., we set $w_{11}=1$, $w_{12}=0.1$, $w_{23}=0.1$, and $w_{24}=1$). 
As expected, utilizing heterogeneity in the value of information allows network to converge a close neighborhood of the actual average of the target quantities sensed by agents. 

For the second example, we consider a networked multiagent system tracking a moving target shown in Figure \ref{fig:dynamic_map}. 
Each agent has a sensing radius, where the value of information obtained by the agents decrease as the target moves away from them. 
In this study, we set $\alpha=20$, $\gamma=150$, and $\sigma=0.1$ in (\ref{eq:ap_weight_1}) and (\ref{eq:ap_weight_2}). 
Figure \ref{fig:dynamic_rec} shows the network response both with identical value of information and with heterogeneous value of information. 
Once again, utilizing heterogeneity in the value of information allows the network to sense the actual trajectory with improved tracking accuracy. 

\begin{figure}[t!] \hspace{-0cm} \includegraphics[scale=0.44]{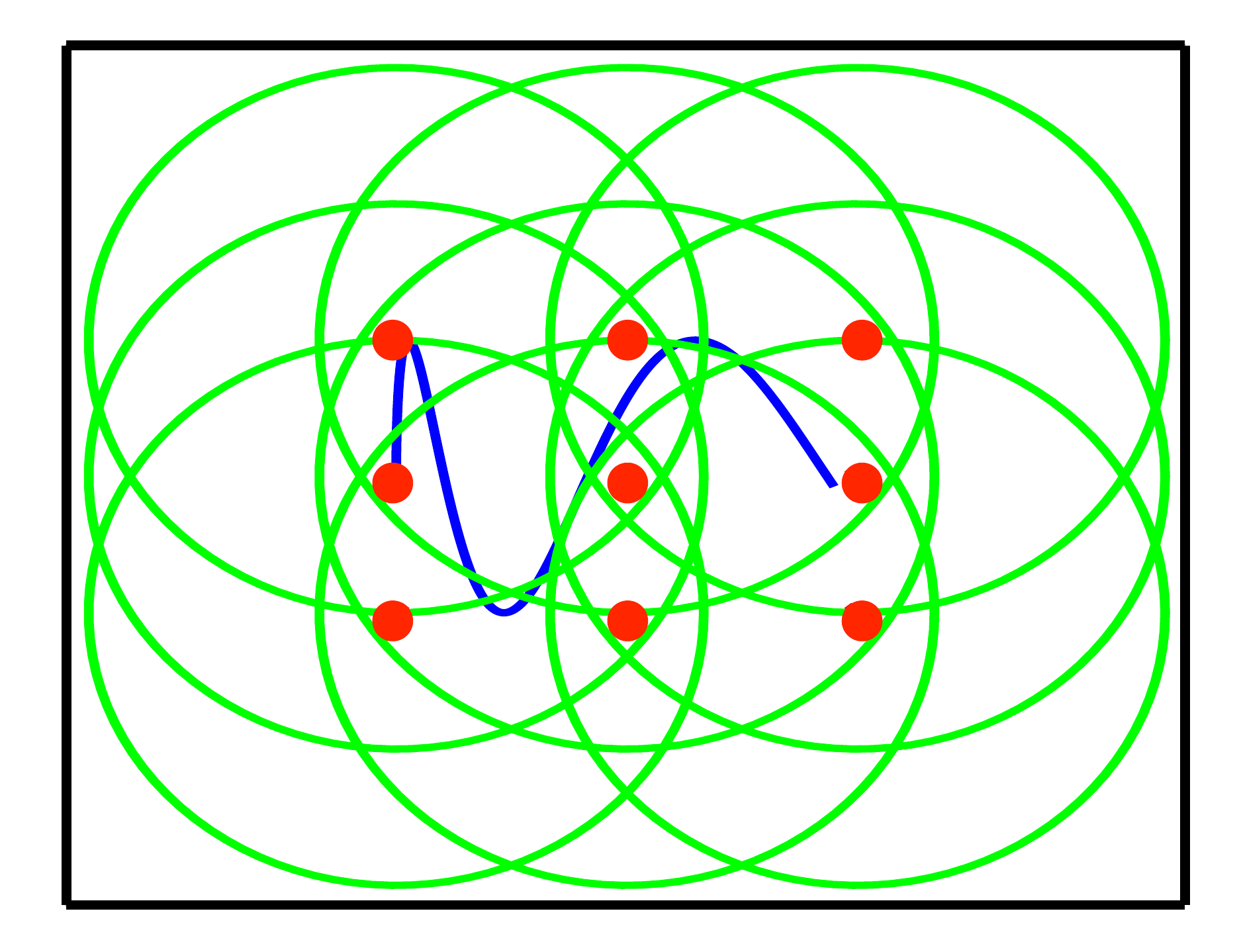} \vspace{-0.5cm}
\caption{Active-passive networked multiagent system with one non-stationary target and nine fixed agents (dots denote the agents, circles denote the sensing radius of agents, and solid line denote the actual target trajectory).} \vspace{-0.35cm}
\label{fig:dynamic_map}
\end{figure}

\begin{figure}[t!] \hspace{-0cm} \includegraphics[scale=0.44]{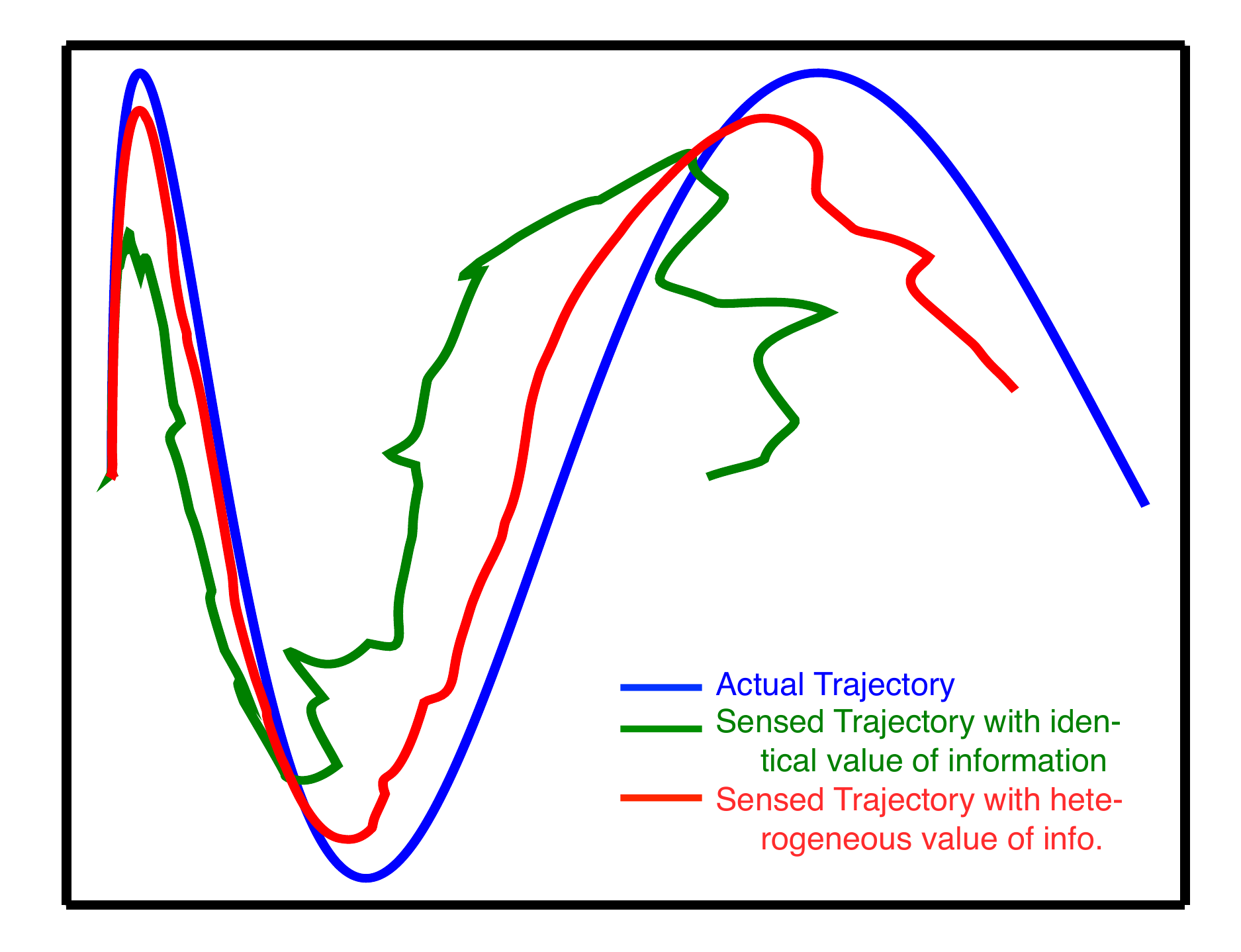} \vspace{-0.5cm}
\caption{Response of the networked multiagent system in Figure \ref{fig:dynamic_map} with identical value of information and heterogeneous value of information.} \vspace{-0.35cm}
\label{fig:dynamic_rec}
\end{figure}


\section{Conclusion}

We utilized the active-passive networked multiagent systems approach to develop a distributed sensing framework that accounts for the heterogeneity in agents' sensing capability measured by the value of information. 
It is shown that the states of all agents converge to an adjustable neighborhood of the weighted average of the sensed exogenous inputs by the active agents when there exists time-variation in both agents' sensing capability and the sensed exogenous inputs. 
In addition, we discussed several cases when agents' sensing capability and the exogenous inputs are time-invariant, which yields asymptotic stability of the error dynamics between the states of all agents and the weighted average of the sensed exogenous inputs. 
Illustrative examples indicated that utilizing heterogeneity allows the networked multiagent system to achieve better distributed sensing performance. \\

\bibliography{DanielPetersonBib,TanselYucelenBib} 

\end{document}

%% file: macro.tex
\newfont{\smalll}{cmr8}

\def\IR{\mathbb{R}}
\def\IS{\mathbb{S}}
\def\zero{\mathbf{0}}
\def\one{\mathbf{1}}

\def\diag{\mathrm{diag}}
\def\eye{\mathrm{I}}

\def\IS{\hbox{I\hskip-.1em S}}
\def\IC{\hbox{C\hskip-
.5em\raise.5ex\hbox{$\scriptscriptstyle\mid$}}\ }
\def\Ic{\hbox{\smalll C\hskip-
.5em\raise.3ex\hbox{$\scriptscriptstyle\mid$}}\ }

\def\T={\buildrel {\scriptscriptstyle\triangle} \over =}

\def\diag{\mathop{\rm diag}}

\def\block-diag{\mathop{\rm block{\scriptstyle -}diag}}

\def\pmbb#1{\setbox0=\hbox{#1}\raise 0.5ex\box0}

\def\norm#1{\|#1\|}

\newcommand{\mT}{^\mathrm{T}}

\newcommand{\rom}{\mathrm}

\renewcommand\theenumi{\roman{enumi}}
\renewcommand\theenumii{\alph{enumii}}
\renewcommand\theenumiii{\arabic{enumiii}}
\renewcommand\theenumiv{\Alph{enumiv}}

\def\sc{s_{\rm c}}

\def\sc{s_{\rm c}}